\documentclass[a4paper,12pt]{article}
\usepackage{graphicx} 
\usepackage[T2A]{fontenc}
\usepackage[utf8]{inputenc}
\usepackage{dsfont}
\usepackage{amsfonts}
\usepackage{amssymb}
\usepackage{amsmath}
\usepackage{amsthm}
\usepackage{comment}
\usepackage{graphicx}
\usepackage{todonotes}
\usepackage{tikzsymbols}
\usepackage{enumitem}
\usepackage{authblk}
\usepackage{hyperref}
\usepackage{mathdots}
\usepackage{mathtools}

\usepackage{bm}

\usepackage[most]{tcolorbox}

\makeatletter
\DeclareRobustCommand{\cev}[1]{%
  \mathpalette\do@cev{#1}%
}
\newcommand{\do@cev}[2]{%
  \fix@cev{#1}{+}%
  \reflectbox{$\m@th#1\vec{\reflectbox{$\fix@cev{#1}{-}\m@th#1#2\fix@cev{#1}{+}$}}$}%
  \fix@cev{#1}{-}%
}
\newcommand{\fix@cev}[2]{%
  \ifx#1\displaystyle
    \mkern#23mu
  \else
    \ifx#1\textstyle
      \mkern#23mu
    \else
      \ifx#1\scriptstyle
        \mkern#22mu
      \else
        \mkern#22mu
      \fi
    \fi
  \fi
}

\newtheorem{thm}{Theorem}
\newtheorem*{thm*}{Theorem}
\newtheorem{lem}[thm]{Lemma}

\newtheorem{prop}[thm]{Proposition} 
\newtheorem{defn}[thm]{Definition}
\newtheorem*{defn*}{Definition}
\newtheorem*{prob*}{Problem}

\newtheorem*{remark*}{Remark}
\newtheorem*{claim*}{Claim}

\def\bbf{\mathbb{F}}
\def\fq{\bbk}
\def\fqt{\fq[t]}

\def\fqttt{\fq\!\left(\!\left(t^{-1}\right)\!\right)}
\def\ft{\fq[t]}
\def\ftt{\fq(t)}
\def\fttt{\fq\!\left(\!\left(t^{-1}\right)\!\right)}
\def \q{2}
\def\sep{\,:\,}

\def\defeq{:=}
\def\eqdef{=:}

\def\bbn{\mathbb{N}}
\def\bbz{\mathbb{Z}}
\def\bbq{\mathbb{Q}}
\def\bbr{\mathbb{R}}
\def\bbk{\mathbb{K}}

\def\calf{\mathcal{F}}
\def\cals{\mathcal{S}}
\def\calg{\mathcal{G}}

\def\defeq{:=}

\newcommand{\HankelPolynomial}{D}
\newcommand{\HankelDeterminant}{\Delta}
\newcommand {\ignore}[1]  {}

\begin{document}

\author{Dmitry Gayfulin}
\author{Erez Nesharim}
\affil{{Department of Mathematics, Technion -- Israel Institute of Technology, Haifa.} \textit{gamak.57.msk@gmail.com}, \textit{ereznesh@gmail.com}}



\title{Sums of Laurent series with bounded partial quotients}

\maketitle

\begin{abstract}
Hall \cite{hall} proved that every real number is the sum of an integer and two real numbers whose partial quotients are at most $4$. Cusick \cite{cusick} proved that every real number is the sum of an integer and two real numbers whose partial quotients are at least $2$. In a recent paper, the authors proved that every real number is the sum of two real numbers whose partial quotients diverge. 
In this paper, we prove an analogue of these results for Laurent series. 
\end{abstract}

\section{Introduction}
For every real number $\alpha$, let $a_n(\alpha)$ be the $n$th partial quotient of $\alpha$, whenever its continued fraction expansion is of length 
at least $n$. For every $k\geq2$, let

\begin{equation}\label{eq:FS}
\begin{aligned}
F(k)&\defeq\left\{\alpha\in\bbr\sep a_n(\alpha)\le k \textrm{ for every $n\geq1$}\right\},\\
S(k)&\defeq\left\{\alpha\in\bbr\sep a_n(\alpha)\ge k \textrm{ for every $n\geq1$ for which $a_n(\alpha)$ is defined}\right\},
\end{aligned}
\end{equation}
and let
\begin{equation}\label{eq:G}
    G\defeq\left\{\alpha\in\bbr\sep a_n(\alpha)\to\infty\right\}\cup\bbq\,.
\end{equation}
Hall \cite{hall} proved that 
\begin{equation}\label{thm:hall}
F(4)+F(4) = \bbr\,.
\end{equation}
This result incited many investigations, which, among other results, show that (see \cite{astel} for a friendly introduction to these results)
\begin{equation}\label{eq:hall}
F(k)+F(m) = \bbr\quad\iff\quad k+m\geq 7\,.
\end{equation}
Complementarily, Cusick \cite{cusick} proved that 
\begin{equation}\label{eq:cusick}
S(2)+S(2) = \bbr\,,
\end{equation}
and Shulga \cite{shulga} extended \eqref{eq:cusick} by showing that for every $3\leq k \leq m\leq (k-1)^2$ the sum of $S(k)$ and $S(m)$ contains an interval
\begin{equation}\label{eq:shulga}
S(k) + S(m)\supseteq \left[0,\frac{1}{k-1}\right].
\end{equation}
This result is achieved through an explicit algorithm that is reviewed in Section \ref{sec:shulga}. Continuing the analysis of Shulga's algorithm, the authors \cite{GN} proved that
\begin{equation}\label{eq:GN}
    G+G=\bbr\,.
\end{equation}
In this paper, we prove analogues of these results for Laurent series.

Let $\fq$ be an arbitrary field, $\ft$ its ring of polynomials, and $\ftt$ its field of rational functions. The field $\fttt$ is the field of all Laurent series in the variable $t$ and coefficients from $\fq$
\begin{equation}\label{eq:theta}
\alpha = \sum_{n=-\infty}^\infty \alpha_n t^{-n}\,,
\end{equation}
where $\alpha_n\in\fq$ for every $n\in\bbz$ and only finitely many of the coefficients of positive powers of $t$ are nonzero. 

For every Laurent series $\alpha\in\fttt$, let $a_n(\alpha)$ be the $n$th partial quotient of $\alpha$, whenever its continued fraction expansion is of length 
at least $n$, and let 
$$d_n(\alpha)\defeq\deg\left(a_n\right)$$
(see Section \ref{prelim_sec} for more background on continued fractions of Laurent series). For every $k\geq1$, we consider the following analogues of the sets in \eqref{eq:FS} and \eqref{eq:G}:
\begin{align*}
\calf(k)&\defeq\left\{\alpha\in\fttt\sep d_n(\alpha)\le k \textrm{ for every $n\geq1$}\right\},\\
\cals(k)&\defeq\left\{\alpha\in\fttt\sep d_n(\alpha)\ge k \textrm{ for every $n\geq1$ for which $a_n(\alpha)$ is defined}\right\},
\end{align*}
and
\begin{equation}
    \calg\defeq\left\{\alpha\in\fttt\sep d_n(\alpha)\to\infty\right\}\cup\ftt\,.
\end{equation}
Our main theorems are Laurent series analogues of \eqref{eq:hall}, \eqref{eq:shulga} and \eqref{eq:GN}:
\begin{thm}
\label{hall_thm}
If $\fq\neq\bbf_2$ then the sumset of $\calf(1)$ with itself covers $\fttt$
\begin{equation}
\label{laurent_representation}
\calf(1)+\calf(1)=\fttt.
\end{equation}
If $\fq=\bbf_2$ then \eqref{laurent_representation} does not hold, but
\begin{equation}
\label{laurent_representation21}
\calf(2)+\calf(1)=\fttt.
\end{equation}
\end{thm}

\begin{thm}\label{thm:S_analogue}
For every field $\fq$ and every $1\le k\leq m$ the sum of the sets $\cals(k)$ and $\cals(m)$ satisfies
\begin{equation*}
\cals(k)+\cals(m) \supseteq \left\{\alpha\in\fttt\sep\alpha_n=0 \textrm{ for every $1\leq n < \max\left\{k,\frac{m-1}{2}\right\}$}\right\}.
\end{equation*}
\end{thm}

\begin{thm}\label{thm:G_analogue}
For every field $\fq$ the sumset of $\calg$ with itself covers $\fttt$
\begin{equation*}
\calg+\calg = \fttt.
\end{equation*}
\end{thm}

In fact, Shulga's algorithm behaves more nicely in the Laurent series setup. Therefore, we can prove a stronger version of Theorem \ref{thm:G_analogue}. Let the set of all Laurent series whose degrees of partial quotient are strictly increasing be
\begin{equation*}
\calg'\defeq\left\{\alpha\in\bbr\sep d_{n}(\alpha)<d_{n+1}(\alpha) \textrm{ for every $n\geq1$ such that $a_{n+1}(\alpha)$ is defined}\right\}.
\end{equation*}
We will show that Theorem \ref{thm:G_analogue} holds with $\calg$ replaced with its subset $\calg'$:

\begin{thm}\label{thm:Gprime_analogue}
For every field $\fq$ the sumset of $\calg'$ with itself covers $\fttt$
\begin{equation*}
\calg'+\calg' = \fttt.
\end{equation*}
\end{thm}

It is interesting to note that the following real analogue of Theorem \ref{thm:Gprime_analogue} is currently not known.
\begin{prob*}
Let 
\begin{equation*}
    G'\defeq\left\{\alpha\in\bbr\sep a_n(\alpha)<a_{n+1}(\alpha)\textrm{ for every $n\geq1$ such that $a_{n+1}(\alpha)$ is defined}\,\right\}.
\end{equation*}
Is it true that the sumset of $G'$ with itself covers $\bbr$
\begin{equation*}
G'+G' = \bbr\,?
\end{equation*}
\end{prob*}

\section{Basic properties of continued fractions of Laurent series}
\label{prelim_sec}
In this section, we recall some basic facts about continued fractions of Laurent series that we use in the article. For a more detailed background, we refer to \cite[Section 9]{thakur}. 

For $\alpha$ as in \eqref{eq:theta}, let its \emph{degree} and \emph{absolute value} 
be 
$$\deg\alpha\defeq\sup\left\{ -n\sep \alpha_n\neq0\right\}$$
and 
$$\left|\alpha\right|=\q^{\deg\alpha}\,,$$ 
respectively (where $\deg0=-\infty$ and $\left|0\right|=0$). 
It is well known that with this absolute value $\ftt$ is embedded in $\fttt$ as the subfield of all Laurent series with an eventually periodic sequence of coefficients.

Every $\alpha\in\fttt$ has a unique continued fraction expansion
$$
\alpha=a_0+\cfrac{1}{a_1+\cfrac{1}{a_2+\cdots\vphantom{\cfrac{1}{1}}}}\eqdef[a_0;a_1,a_2,\dots]\,,
$$
where $a_n\in\fqt$ for every $n\geq0$, and $d_n\defeq\deg(a_n)$ satisfies $d_n\geq1$ for every $n \geq 1$. This expansion is unique for every $\alpha$, and it is finite if and only if $\alpha$ is a rational function. The convergents of a continued fraction are the rational functions
$$
\frac{p_n}{q_n}:=[a_0;a_1,a_2,\ldots,a_n]
$$
whose numerators and denominators are defined by the recurrence relations
\begin{equation*}
\begin{aligned}
    p_0&=a_0\,,\quad &p_1&=a_1a_0+1\,,\quad
    &p_{n}&=a_{n}p_{n-1}+p_{n-2}\,,\quad\textrm{and}\\
    q_0&=1\,,\quad&q_1&=a_1\,,\quad
    &q_{n}&=a_{n}q_{n-1}+q_{n-2}\,.
\end{aligned}
\end{equation*}
For any $n\ge0$ we consider $a_n(\cdot)$, $d_n(\cdot)$, $p_n(\cdot)$ and $q_n(\cdot)$ as functions that return the $n$th partial quotient of the continued fraction expansion and its degree, the numerator and denominator of its $n$th convergent, respectively. If $\alpha$ is rational, then $n$ should not exceed the length of the continued fraction expansion of $\alpha$; otherwise, these functions are undefined.

The following properties are well known:
\begin{enumerate}
\item The degrees of partial quotients and denominators are related by 
\begin{equation}\label{eq:common1}
\deg(q_n)=\deg(q_{n-1}) + d_n\,
\end{equation}
whenever $a_n$ is defined.
\item For every $n\geq0$, if $a_{n+1}$ is defined then 
\begin{equation}\label{eq:common3}
\left|\alpha-\frac{p_n}{q_n} \right|=\frac{1}{|q_n|\,|q_{n+1}|}\,.
\end{equation}
\item Let $p/q\in\ftt$. Then
\begin{equation}\label{eq:common4}
    \left|\alpha-\frac{p}{q} \right|<\frac{1}{|q|^2} \quad\Longrightarrow\quad \frac pq=\frac {p_n}{q_n}\,\textrm{ for some $n\geq0$}\,. 
\end{equation}
\ignore{\item
\begin{equation}\label{eq:common5}
    \frac{q_n}{q_{n-1}} = [a_n;a_{n-1},\ldots,a_{1}]\,.
\end{equation}}
\item For every $h\geq1$, let $\HankelPolynomial_h$ be the polynomial with $2h-1$ variables $x_1,\ldots,x_{2h-1}$ which is given by the determinant
\begin{equation}\label{eq:tm_hankel}
	\HankelPolynomial_h(x_1,\ldots,x_{2h-1}) \defeq 
	\begin{vmatrix}
		x_{1}&x_{2}&\cdots &x_{h}\\
		x_{2}&\iddots&\iddots&x_{h+1}\\
		\vdots&\iddots&\iddots&\vdots\\
		x_{h}&x_{h+1}&\cdots&x_{2h-1}
	\end{vmatrix},
\end{equation}
and let
\begin{equation*}
\HankelDeterminant_h=\HankelDeterminant_h(\alpha)\defeq\HankelPolynomial_h(\alpha_1,\ldots,\alpha_{2h-1})\,.
\end{equation*}
Then
\begin{equation}\label{eq:det_lem}
    \begin{aligned}
        \HankelDeterminant_h\ne 0 
        \quad & \iff \quad \min_{p,q\in\fqt,\;0\leq\deg q<h} | q\theta - p |=2^{-h}  \\
        & \iff \quad 
        h=\deg(q_n) \,\textrm{ for some $n\geq1$}\,\\
    \end{aligned}
\end{equation}
where the first equivalence follows by rewriting $| q\theta - p |=2^{-h}$ as a system of linear equations, and the second equivalence follows from \eqref{eq:common3} and \eqref{eq:common4}.
It will also be convenient to let $D_0$ be the constant polynomial $1$.
\end{enumerate}

Note that \eqref{eq:common4} is the Laurent series analogue of Legendre's theorem, and its converse holds by \eqref{eq:common3}. Also note that \eqref{eq:common1} and \eqref{eq:det_lem} imply that
\begin{equation}\label{eq:913}
    \alpha\in\calf(k) \quad\iff\quad \begin{aligned}
        &\textrm{for every $h\geq k$, if $\HankelDeterminant_{h-i}=0$ for every $1\leq i< k$ then $\HankelDeterminant_{h}\neq0$}\,.
    \end{aligned}
\end{equation}

\ignore{Equations \eqref{eq:common3} and \eqref{eq:common4} have a nice interpretation in terms of Hankel determinants that we now recall. For every $h\geq1$, let $\HankelPolynomial_h$ be a polynomial with $2h-1$ variables $x_1,\ldots,x_{2h-1}$ which is given by
\begin{equation}\label{eq:tm_hankel}
	\HankelPolynomial_h(x_1,\ldots,x_{2h-1}) \defeq 
	\begin{vmatrix}
		x_{1}&x_{2}&\cdots &x_{h}\\
		x_{2}&\iddots&\iddots&x_{h+1}\\
		\vdots&\iddots&\iddots&\vdots\\
		x_{h}&x_{h+1}&\cdots&x_{2h-1}
	\end{vmatrix}.
\end{equation}
It will also be convenient to let $D_0$ be the constant polynomial $1$.
\begin{thm*}
\label{det_lem}
Let $\alpha\in\fttt$ and $h\geq1$. Then 
$$\HankelPolynomial_h(\alpha_1,\ldots,\alpha_{2h-1})\ne 0 \quad \iff \quad h=\deg(q_n) \,\textrm{ for some $n\geq1$}\,.$$
\end{thm*}}

\section{An analogue of Hall's theorem}

The proof of Theorem \ref{hall_thm} is based on analysing the determinant \eqref{eq:tm_hankel}:

\begin{lem}\label{lem:inspection}
    For every $h\geq1$ there exists a degree $h+1$ homogeneous polynomial of $2h$ variables $P_h$ such that
    \begin{equation}\label{delta_np1_1eq}
        \HankelPolynomial_{h+1}(x_1,\ldots,x_{2h+1})=x_{2h+1}\HankelPolynomial_h(x_1,\ldots,x_{2h-1})+P_h(x_1,\ldots,x_{2h}).
    \end{equation}
    Moreover, if $\fq=\bbf_2$ then there exists a degree $h+1$ homogeneous polynomial $Q_h$ of $2h-1$ variables such that
    \begin{equation}\label{delta_np1_2eq}
        P_{h}(x_1,\ldots,x_{2h})=x_{2h}^2\HankelPolynomial_{h-1}(x_1,\ldots,x_{2h-3})+Q_h(x_1,\ldots,x_{2h-1}).
    \end{equation}
\end{lem}
\begin{proof}
    The proof follows by inspection of the determinant
    \begin{equation}\label{eq:hankel1}
    	\HankelPolynomial_{h+1}(x_1,\ldots,x_{2h+1}) = 
    	\begin{vmatrix}
    		x_{1}&x_{2}&\cdots &x_{h}&x_{h+1}\\
    		x_{2}&\iddots&\iddots&x_{h+1}&x_{h+2}\\
    		\vdots&\iddots&\iddots&\vdots&\vdots\\
    		x_{h}&x_{h+1}&\cdots&x_{2h-1}&x_{2h}\\
            x_{h+1}&x_{h+2}&\cdots&x_{2h}&x_{2h+1}\\
    	\end{vmatrix}.
    \end{equation}
Indeed, note that $x_{2h+1}$ appears in \eqref{eq:hankel1} only once, in the lower-right corner. So, expanding $\HankelPolynomial_{h+1}(x_1,\ldots,x_{2h+1})$ by, for example, the last row, gives \eqref{delta_np1_1eq}. To see \eqref{delta_np1_2eq}, it is useful to think of $\HankelPolynomial_{h+1}$ as a sum of $(h+1)!$ monomials of degree $h+1$, where the summands are in bijection with permutations of $h+1$ elements. Note that $x_{2h}$ appears in \eqref{eq:hankel1} exactly twice, in the one-before-last entry in the last row and column. Therefore, every permutation that contains the term $x_{2h}$ exactly once, can be paired with the permutation that arises by taking its entries reflected along the main diagonal. Since \eqref{eq:hankel1} is invariant under reflection along the main diagonal, the monomials attached to these two permutations are equal. Since $\fq=\bbf_2$, they cancel each other and the polynomial $\HankelPolynomial_{h+1}$ has no monomials which are divisible by $x_{2h}$ but not by $x_{2h}^2$.
\end{proof}

Let $\alpha\in\fttt$. Without loss of generality, we may assume that $a_0(\alpha)=0$. 
We start with a proof of the first statement of Theorem \ref{hall_thm}. 
In view of \eqref{eq:913}, to show \eqref{laurent_representation} it is sufficient to show that there exist two infinite sequences $\{\beta_n\}_{n\in\bbn},\{\gamma_n\}_{n\in\bbn}\in\fq^\bbn$ such that for every $n\ge 1$ one has 
\begin{equation}
\label{sum_cond}
\alpha_n=\beta_n+\gamma_n\,,
\end{equation}
and all the determinants of the form \eqref{eq:tm_hankel} are non-zero:
\begin{align}
\label{nonzero_cond_beta}
&\HankelPolynomial_n(\beta_1,\ldots,\beta_{2n-1})\neq0\,,\quad\textrm{ and }
\\
\label{nonzero_cond_gamma}
&\HankelPolynomial_n(\gamma_1,\ldots,\gamma_{2n-1})\neq0\,.
\end{align}
Indeed, then 
$$
\beta:=\sum\limits_{n=1}^{\infty}\beta_n t^{-n}\quad\textrm{and}\quad
\gamma:=\sum\limits_{n=1}^{\infty}\gamma_n t^{-n}
$$
satisfy $\alpha=\beta+\gamma$ and $\beta,\gamma\in\calf(1)$.
We proceed by induction.

If $n=1$, choose any nonzero $\beta_1$ which is not $\alpha_1$, and $\gamma_1=\alpha_1-\beta_1$. This is possible since $\fq\neq\bbf_2$. Now 
suppose that $n\geq1$ and $\beta_i$ and $\gamma_i$ are defined for every $1\le i\le 2n-1$ and satisfy \eqref{nonzero_cond_beta} and \eqref{nonzero_cond_gamma}. Choose any $\beta_{2n}$, and define $\gamma_{2n}=\alpha_{2n}-\beta_{2n}$. Choose any $\beta_{2n+1}$ not in
\begin{equation}\label{eq:pigeonhole}
    \left\{-\frac{P_n(\beta_1,\ldots,\beta_{2n})}{\HankelPolynomial_{n}(\beta_1,\ldots,\beta_{2n-1})}\,,
    \alpha_{2n+1}+\frac{P_n(\gamma_1,\ldots,\gamma_{2n})}{\HankelPolynomial_{n}(\gamma_1,\ldots,\gamma_{2n-1})}\right\} \,.
\end{equation}
The set in \eqref{eq:pigeonhole} is well defined by the induction assumptions \eqref{nonzero_cond_beta} and \eqref{nonzero_cond_gamma}, and it does not cover $\fq$ since $|\fq|\geq3$. Define $\gamma_{2n+1} = \alpha_{2n+1} - \beta_{2n+1}$. It follows from \eqref{delta_np1_1eq} that $\HankelPolynomial_{n+1}(\beta_1,\ldots,\beta_{2n+1})$ and $\HankelPolynomial_{n+1}(\gamma_1,\ldots,\gamma_{2n+1})$ are nonzero.
This concludes the induction step.

We now turn to the case $\fq=\bbf_2$. First of all, the property \eqref{laurent_representation} does not hold, because, for example, $t^{-1}\notin\calf(1)+\calf(1)$. Indeed, if $\beta+\gamma = t^{-1}$, then either $\beta_1$ or $\gamma_1$ is zero, and therefore, either $\beta$ or $\gamma$ are not in $\calf(1)$, respectively. 

In view of \eqref{eq:913}, to show \eqref{laurent_representation21} it is sufficient to show that there exist two infinite sequences $\{\beta_n\}_{n\in\bbn},\{\gamma_n\}_{n\in\bbn}\in\fq^\bbn$ such that for every $n\ge 1$ one has \eqref{sum_cond}, and \eqref{nonzero_cond_gamma}, and 
\begin{equation}\label{eq:f2}
    \textrm{if }\,D_{n-1}(\beta_1,\ldots,\beta_{2n-3})=0\, \textrm{ then } D_{n}(\beta_1,\ldots,\beta_{2n-1})\neq0\,.
\end{equation}

\ignore{
For $n=1$, if $\alpha_1=1$, we put 
\begin{equation*}
\HankelPolynomial_{\beta}(2)=
\begin{vmatrix}
0 & 1 \\
1 & \alpha_3+\alpha_2
\end{vmatrix}
,\quad
\HankelPolynomial_{\gamma}(2)=
\begin{vmatrix}
1 & \alpha_2+1 \\
\alpha_2+1 & \alpha_2
\end{vmatrix}
.
\end{equation*}
If $\alpha_1=0$, we put 
\begin{equation*}
\HankelPolynomial_{\beta}(2)=
\begin{vmatrix}
1 & \alpha_2 \\
\alpha_2 & \alpha_3+1
\end{vmatrix}
,\quad
\HankelPolynomial_{\gamma}(2)=
\begin{vmatrix}
1 & 0 \\
0 & 1
\end{vmatrix}
.
\end{equation*}
}

For $n=1$, choose $\gamma_1=1$ and $\beta_1=\alpha_1+1$. 
\ignore{For $n=2$, if $\beta_1=0$ then choose $\beta_2=1$ and $\gamma_2=\alpha_2+1$, and then choose $\beta_3=\alpha_2+\alpha_3$ and $\gamma_3=\alpha_2$. Otherwise, if $\beta_1=1$ then choose $\beta_2$ arbitrary and $\gamma_2=\alpha_2+\beta_2$, and then choose $\beta_3=\gamma_2$ and $\gamma_3=\gamma_2+1$. It is easy to verify that \eqref{sum_cond}, \eqref{nonzero_cond_gamma} and \eqref{eq:f2} hold for $n=2$.}
Now suppose that $n\geq 1$ and that $\beta_i$ and $\gamma_i$ are defined for every $1\le i\le 2n-1$ and satisfy \eqref{nonzero_cond_gamma} and \eqref{eq:f2}. If 
\begin{equation}\label{eq:case1}
    \HankelPolynomial_n(\beta_1,\ldots,\beta_{2n-1})=0
\end{equation}
then choose 
\begin{equation}\label{eq:pigeonhole2}
\beta_{2n} \neq \frac{Q_n(\beta_1,\ldots,\beta_{2n-1})}{\HankelPolynomial_{n-1}(\beta_1,\ldots,\beta_{2n-3})}
\end{equation}
and put $\gamma_{2n}=\alpha_{2n}+\beta_{2n}$. Then choose 
\begin{equation}\label{eq:pigeonhole3}
\gamma_{2n+1} \neq \frac{P_n(\gamma_1,\ldots,\gamma_{2n})}{\HankelPolynomial_{n}(\gamma_1,\ldots,\gamma_{2n-1})}
\end{equation}
and $\beta_{2n+1} = \alpha_{2n+1}+\gamma_{2n+1}$. We recall that if $n=1$ then $D_{n-1}=1$ by the definition. Note that \eqref{eq:pigeonhole2} is well defined since the induction assumption \eqref{eq:f2} and \eqref{eq:case1} imply that $\HankelPolynomial_{n-1}(\beta_1,\ldots,\beta_{2n-3})\neq0$, while \eqref{eq:pigeonhole3} is well defined by the induction assumption \eqref{nonzero_cond_gamma}. Recall that $\fq=\bbf_2$, so $\beta_{2n}^2=\beta_{2n}$. Therefore, Lemma \ref{lem:inspection} implies that $\HankelPolynomial_{n+1}(\beta_1,\ldots,\beta_{2n+1})$ and $\HankelPolynomial_{n+1}(\gamma_1,\ldots,\gamma_{2n+1})$ are nonzero.

Otherwise, if \eqref{eq:case1} does not hold, choose $\beta_{2n}$ arbitrarily and $\gamma_{2n}=\alpha_{2n}+\beta_{2n}$, and then choose $\gamma_{2n+1}$ as in \eqref{eq:pigeonhole3}, and $\beta_{2n+1}=\alpha_{2n+1}+\gamma_{2n+1}$. Lemma \ref{lem:inspection} implies that $\HankelPolynomial_{n+1}(\gamma_1,\ldots,\gamma_{2n+1})\neq0$.
This finishes the induction argument and the proof of Theorem \ref{hall_thm}.

\section{An analogue of Shulga's algorithm}\label{sec:shulga}
In \cite{shulga} Shulga established a nice and simple algorithm that decomposes real numbers into a sum of two real numbers with growing partial quotients. Theorems \ref{thm:S_analogue} and \ref{thm:Gprime_analogue} are immediate corollaries of an analogue of this algorithm for Laurent series and its properties, which are summarised in Theorem \ref{laurent_prop}. We first recall the real version of this algorithm, and some of its properties.

\begin{defn}[Shulga]\label{def:shulga}
For every $\alpha\in[0,1]$, define sequences $\{b_n\}_{n\geq1}$ and $\{c_n\}_{n\geq1}$, and real numbers $\beta$ and $\gamma$ as follows:
\begin{enumerate}
\item For every $n\geq0$, having the integers $b_j$ and $c_j$ defined for every $1\leq j\leq n$, if 
$$\alpha = [0;b_1,\ldots,b_{n}] + [0;c_1,\ldots,c_{n}]\,,$$ 
let
\begin{equation*}
    \beta\defeq [0;b_1,\ldots,b_{n}]\,,\quad\textrm{and} \quad \gamma\defeq [0;c_1,\ldots,c_{n}]\,,
\end{equation*}
and stop, otherwise, let
\begin{equation*}
    \begin{aligned}
        b_{n+1}\defeq a_{n+1}(\alpha-[0;c_1,\ldots,c_{n}])+1\,,\quad\textrm{and} \quad
        c_{n+1}\defeq a_{n+1}(\alpha-[0;b_1,\ldots,b_{n+1}])\,.
    \end{aligned}
\end{equation*}
\item If the algorithm does not stop after finitely many steps, let
\begin{equation*}
    \beta\defeq [0;b_1,b_2,\ldots]\,,\quad \textrm{and} \quad \gamma\defeq [0;c_1,c_2,\ldots]\,.
\end{equation*}
\end{enumerate}
\end{defn}
The properties of the sequences $\{b_n\}_{n\geq1}$ and $\{c_n\}_{n\geq1}$ provided by Definition \ref{def:shulga} were studied in 
\cite{shulga} and 
\cite{GN}, and the main results are summarised in the following statement:
\begin{thm}[\cite{shulga} and \cite{GN}]
\label{real_alg_prop}
Let $\alpha\in[0,1]$. Then:
\begin{enumerate}
\item
The numbers $\beta$ and $\gamma$ in Definition \ref{def:shulga} are well defined and satisfy 
\begin{equation*}
    \alpha=\beta+\gamma\,.  
\end{equation*} 
Moreover, for every $n\geq1$, if $b_n$ and $c_n$ of Definition \ref{def:shulga} are defined, then 
\begin{equation*}
\begin{aligned}
    b_i = a_i(\alpha-[0;c_1,\ldots,c_n])\,\quad\textrm{and}\quad 
    c_i = a_i(\alpha-[0;b_1,\ldots,b_n])\,\quad\textrm{for every $1\leq i\leq n$}\,.
\end{aligned}
\end{equation*}

\item
If $b_1$ and $c_1$ are defined, then 
$$
c_1\ge b_1(b_1-1)>1\,,
$$
for every $n\ge 2$, if $b_n$ and $c_n$ are defined, then
\begin{equation*}
    c_n\ge b_n+1\,,
\end{equation*}
and for every $n\ge 1$, if $b_n$, $b_{n+1}$ and $b_{n+2}$ are defined, then 
\begin{align*}
b_n&\ge n\,,\\
b_{n+1}&\ge b_n\,,\\
b_{n+2}&\ge b_n+1\,,
\end{align*}
respectively.
\item
If $\alpha\in\mathbb{Q}$ then $\beta$ and $\gamma$ are also rational, i.e., the algorithm in Definition \ref{def:shulga} stops after finitely many steps.
\end{enumerate}
\end{thm}
    
We are now ready to define the Laurent series analogue of Shulga's algorithm: 
\begin{defn}\label{def:shulga_an}
For any $\alpha\in\fqttt$ such that $a_0(\alpha)=0$ define the sequences $\{b_n\}_{n\geq1}$ and $\{c_n\}_{n\geq1}$, and the Laurent series $\beta$ and $\gamma$ as follows:
\begin{enumerate}
\item
For every $n\geq0$, having the integers $b_j$ and $c_j$ defined for every $1\leq j\leq n$, if
\begin{equation*}
\alpha=[0;b_1,\ldots,b_{n}]+[0;c_1,\ldots,c_{n}]\,,
\end{equation*}
then let
$$\beta\defeq[0;b_1,\ldots,b_{n}]\,, \quad \textrm{and}\quad \gamma\defeq[0;c_1,\ldots,c_{n}]\,,$$
and stop. Otherwise, let
\begin{equation*}
b_{n+1} \defeq a_{n+1}(\alpha-[0;c_1,\ldots,c_{n}])\,.
\end{equation*}
If 
\begin{equation*}
\alpha=[0;b_1,\ldots,b_{n+1}]+[0;c_1,\ldots,c_{n}]\,,
\end{equation*}
then let 
$$\beta\defeq[0;b_1,\ldots,b_{n+1}]\,, \quad \textrm{and}\quad \gamma\defeq[0;c_1,\ldots,c_{n}]\,,$$ 
and stop. Otherwise, let 
\begin{equation*}
c_{n+1} \defeq a_{n+1}(\alpha-[0;b_1,\ldots,b_{n+1}])\,.
\end{equation*}
\item If the algorithm does not stop after finitely many steps, let 
$$\beta\defeq[0;b_1,b_2,\ldots]\,, \quad \textrm{and}\quad \gamma\defeq[0;c_1,c_2,\ldots]\,.$$ 
\end{enumerate}
\end{defn}


The properties of the polynomials $b_n$ and $c_n$ of Definition \ref{def:shulga_an} are established in the following analogue of Theorem \ref{real_alg_prop}:
\begin{thm}
\label{laurent_prop}
Let $\alpha\in\fttt$ be such that $a_0(\alpha)=0$. Then:
\begin{enumerate}
\item\label{item:1}{
The Laurent series $\beta$ and $\gamma$ in Definition \ref{def:shulga_an} are well defined and satisfy 
\begin{equation*}
    \alpha=\beta+\gamma\,.    
\end{equation*} 
Moreover, for every $n\geq1$, if $b_n$ is defined, then
\begin{equation}
\label{ck_well_defined}
    c_i = a_i(\alpha-[0;b_1,\ldots,b_n])\,,\quad\textrm{for every $1\leq i\leq n-1$}\,,
\end{equation}
and if $c_n$ is defined, then
\begin{equation}
\label{bk_well_defined}
    \begin{cases*}
        b_i = a_i(\alpha-[0;c_1,\ldots, c_n])\\
        c_i = a_i(\alpha-[0;b_1,\ldots,b_n])
    \end{cases*},
    \quad\textrm{for every $1\leq i\leq n$}\,.
\end{equation}
}
\item\label{item:2}
{If $c_1$ is defined, then 
\begin{equation}\label{eq:first_digits}
    \deg(c_1)\geq 2\deg(b_1)+1\,,
\end{equation} 
and for every $n\ge 2$, if $b_n$ and $c_n$ are defined, then
\begin{align*}
\deg(b_n) &\geq \deg(c_{n-1})+2\,,\quad\textrm{and}\\
\deg(c_n) &\geq \deg(b_n)+2\,,
\end{align*}
respectively.
}
\item
{If $\alpha=p/q$ is a rational function, then the algorithm stops after at most $\deg(q)/2$ steps, and in particular, $\beta$ and $\gamma$ are also rational functions.
}
\end{enumerate}
\end{thm}


\begin{proof}
We start by proving items \ref{item:1} and \ref{item:2}.
Let $\alpha\in\fttt$ be a fixed Laurent series such that $a_0(\alpha)=0$. For the purpose of this proof, we will adopt the following notation: given $n\geq0$, if $b_1,\ldots,b_n$ are defined as in Definition \eqref{def:shulga_an}, then let
\begin{equation*}
    \frac{p_n}{q_n}\defeq[0;b_1,\ldots,b_n]\,,
\end{equation*}
and if $c_1,\ldots,c_{n}$ are defined as in Definition \ref{def:shulga_an} then let
\begin{equation*}
    \frac{s_n}{t_n}\defeq[0;c_1,\ldots,c_n]\,.    
\end{equation*}
Note that $q_0=t_0=1$. We will prove the following statement which implies both items \ref{item:1} and \ref{item:2}: 

\begin{claim*}
For every $j\geq1$, if $b_j$ is defined, then
\begin{align}
\label{qnp1_ge_2tn}
\deg(t_{j-1})-\deg(q_{j-1}) &< \deg(q_{j}) - \deg(t_{j-1})\,,\quad\textrm{and}\\
\label{eq:qnp1_ge_2tn}
\left|\alpha - \frac{p_{j}}{q_{j}} - \frac{s_{j-1}}{t_{j-1}}\right| &< \frac{1}{|t_{j-1}|^2}\,,
\end{align}
and if $c_j$ is defined, then
\begin{align}
\label{tn1_ge_2qn1_min_tn2}
\deg(q_{j})-\deg(t_{j-1}) &< \deg(t_{j}) - \deg(q_{j})\,,\quad\textrm{and}\\
\label{eq:tn1_ge_2qn1_min_tn2}
\left|\alpha - \frac{p_{j}}{q_{j}} - \frac{s_j}{t_j}\right| &< \frac{1}{|t_j|^2}\,.
\end{align}
\end{claim*}
We first show that the above claim implies item \ref{item:1}. Indeed, 
if $b_n$ is defined, then by \eqref{eq:common3} and \eqref{eq:common4}, equation \eqref{ck_well_defined} is equivalent to \eqref{eq:qnp1_ge_2tn} with $j=n$. Similarly, for every $j$ such that $c_j$ is defined, equations \eqref{qnp1_ge_2tn} and \eqref{tn1_ge_2qn1_min_tn2} imply that 
\begin{equation*}
\deg(t_{j})-\deg(q_{j})\geq \deg(t_{j-1})-\deg(q_{j-1})+2\,,
\end{equation*}
where the inequality follows from the fact that if $x<y<z$ are three integers, then $z-x\geq2$. Since $c_i$ is defined for every $1\leq i \leq j$, this implies that
\begin{equation}
\label{tn_qn_plus2}
\deg(t_{j})-\deg(q_{j})\ge 2j\,.
\end{equation}
In particular, equation \eqref{tn_qn_plus2} gives 
\begin{equation}
\label{tj>qj}
|q_j|<|t_j|\,.
\end{equation}
Therefore, if $c_n$ is defined then \eqref{eq:tn1_ge_2qn1_min_tn2} and \eqref{tj>qj} with $j=n$ give \eqref{bk_well_defined}.

We now verify that the above claim implies item \ref{item:2}. Indeed, assume that $n\geq2$ and that $b_n$ is defined. Recall that \eqref{eq:common1} gives 
\begin{equation*}
    \deg(b_n)=\deg(q_n)-\deg(q_{n-1})\,,\quad\textrm{and}\quad \deg(c_{n-1})=\deg(t_{n-1})-\deg(t_{n-2})\,. 
\end{equation*}
On the other hand, equation \eqref{qnp1_ge_2tn} with $j=n$ and \eqref{tn1_ge_2qn1_min_tn2} with $j=n-1$ give
$$
\deg(q_{n-1})-\deg(t_{n-2}) < \deg(t_{n-1}) - \deg(q_{n-1}) < \deg(q_{n})-\deg(t_{n-1})\,.
$$
Therefore, 
$$
\deg(b_{n})-\deg(c_{n-1}) = (\deg(q_{n}) - \deg(t_{n-1})) - (\deg(q_{n-1}) - \deg(t_{n-2}))\geq2\,,
$$
Similarly, if $c_n$ is defined, then equations \eqref{qnp1_ge_2tn} and \eqref{tn1_ge_2qn1_min_tn2} with $j=n$ imply that
$$
\deg(c_{n})-\deg(b_{n}) = (\deg(t_{n}) - \deg(q_{n})) - (\deg(t_{n-1}) - \deg(q_{n-1}))\geq2\,.
$$
Finally, for $n=1$, equation \eqref{tn1_ge_2qn1_min_tn2} becomes \eqref{eq:first_digits}.

We now prove the above claim by induction on $j$. Assuming that the claim holds for $j$, we will prove it for $j+1$. By \eqref{eq:common4}, \eqref{eq:tn1_ge_2qn1_min_tn2} and \eqref{tj>qj} we find that $\frac{p_j}{q_j}$ is a convergent of $\alpha-\frac{s_j}{t_j}$. Therefore, if $b_{j+1}$ is defined, then by \eqref{eq:common3} we conclude that 
$$
    \left|\alpha - \frac{p_j}{q_j} - \frac{s_j}{t_j}\right| = \frac{1}{|q_j|\,|q_{j+1}|}\,.
$$
Applying \eqref{eq:tn1_ge_2qn1_min_tn2} again gives 
\begin{equation}\label{eq:inductive_step}
    \deg(t_{j})-\deg(q_j)<\deg(q_{j+1})-\deg(t_{j})\,.
\end{equation}
This completes the induction step for \eqref{qnp1_ge_2tn}. 

Since $\frac{p_j}{q_j}$ is a convergent of $\alpha-\frac{s_j}{t_j}$, the definition of $b_{n+1}$ implies that $\frac{p_{j+1}}{q_{j+1}}$ is also a convergent of $\alpha-\frac{s_j}{t_j}$. Therefore, 
\begin{equation}\label{eq:stronger}
    \left|\alpha - \frac{p_{j+1}}{q_{j+1}} - \frac{s_j}{t_j}\right| < \frac{1}{|q_{j+1}|^2}\,.
\end{equation}
Equation \eqref{eq:inductive_step} implies that $|t_j|<|q_{j+1}|$. Therefore, 
\begin{equation}\label{eq:inductive_step2}
    \left|\alpha - \frac{p_{j+1}}{q_{j+1}} - \frac{s_j}{t_j}\right| < \frac{1}{|t_{j}|^2}\,.
\end{equation}
This completes the induction step for \eqref{eq:qnp1_ge_2tn}.

Continuing in the same fashion, equation \eqref{eq:inductive_step2} shows that $\frac{s_j}{t_j}$ is a convergent of $\alpha - \frac{p_{j+1}}{q_{j+1}}$. If $c_{j+1}$ is defined then \eqref{eq:common3} implies that 
$$
    \left|\alpha - \frac{p_{j+1}}{q_{j+1}} - \frac{s_j}{t_j}\right| = \frac{1}{|t_j|\,|t_{j+1}|}\,.
$$
Therefore, equation \eqref{eq:stronger} implies that
\begin{equation*}
    \deg(q_{j+1})-\deg(t_{j}) < \deg(t_{j+1})-\deg(q_{j+1})\,.
\end{equation*}
This completes the induction step for \eqref{tn1_ge_2qn1_min_tn2}.

Since $\frac{s_j}{t_j}$ is a convergent of $\alpha-\frac{p_{j+1}}{q_{j+1}}$, the definition of $c_{j+1}$ implies that $\frac{s_{j+1}}{t_{j+1}}$ is also a convergent, of $\alpha-\frac{p_{j+1}}{q_{j+1}}$. Therefore, 
\begin{equation*}
\left|\alpha - \frac{p_{j+1}}{q_{j+1}} - \frac{s_{j+1}}{t_{j+1}}\right| < \frac{1}{|t_{j+1}|^2}\,.
\end{equation*}
This completes the induction step for \eqref{eq:tn1_ge_2qn1_min_tn2}, and the proof of the above claim.

Finally, consider the case where $\alpha=p/q$ is a rational function.
From \eqref{tn_qn_plus2}, if $n\geq1$ is such that $c_n$ is defined, then we have 
\begin{equation*}
\deg(t_n)-\deg(q_n)\geq 2n\,.
\end{equation*}

On the other hand, item \ref{item:1} means that $\frac{s_n}{t_n}$ is a best approximation of $\alpha-\frac{p_n}{p_n}$. As the degree of the denominator of $\alpha-\frac{p_n}{q_n}$ is at most $\deg(q_n)+\deg(q)$, we obtain the inequality  
\begin{equation*}
\deg(t_n)-\deg(q_n)\leq\deg(q)\,.
\end{equation*}
Therefore, $n\leq \deg(q)/2$, which completes the proof.
\end{proof}

In fact, the bounds in Theorem \ref{laurent_prop}\eqref{item:2} are optimal, as can be seen by the following example:

\begin{prop}
For any field $\fq$ consider the infinite continued fractions
\begin{equation}
\label{beta_gamma_def}
\beta=[0;t,t^5,t^9,t^{13},\ldots]\quad\text{and}\quad
\gamma=[0;t^3,t^7,t^{11},t^{15},\ldots]\,,
\end{equation}
and put $\alpha=\beta+\gamma$. Then the decomposition of $\alpha$ as in Definition \ref{def:shulga_an} gives back $\beta$ and $\gamma$.    
\end{prop}
\begin{proof}
For the purpose of this proof, let $b_n=t^{-(4n-3)}$ and $c_n=t^{-(4n-1)}$ for every $n\geq1$, and let $\frac{p_n}{q_n}$ and $\frac{s_n}{t_n}$ be the $n$th convergents of $\beta$ and $\gamma$, respectively, for every $n\geq0$. 
 To verify the proposition, it is sufficient to show that
\begin{equation}
\begin{split}
\label{bn_cn_example}
b_j=&a_j\left(\alpha-\frac{s_{j-1}}{t_{j-1}}\right),\\
c_j=&a_j\left(\alpha-\frac{p_{j}}{q_{j}}\right)
\end{split}
\end{equation}
for all $j\ge 1$.
We verify \eqref{bn_cn_example} by complete induction on $j$. 
Suppose $n\geq1$ and \eqref{bn_cn_example} holds for every $1\leq j<n$. Due to \eqref{eq:common3} and \eqref{eq:common4}, the first equality in \eqref{bn_cn_example} for $j=n$ is equivalent to 
\begin{equation}
\label{ex_step1_prop}
\left|\alpha-\frac{s_{n-1}}{t_{n-1}}-\frac{p_n}{q_n}\right|<\frac{1}{|q_n|^2}\,.
\end{equation}
Using the ultrametric property of the norm in $\fttt$, we see that
\begin{equation}
\label{ex_step1_prop_ultrametric}
\left|\alpha-\frac{s_{n-1}}{t_{n-1}}-\frac{p_n}{q_n}\right|\le
\max\left\{\left|\beta-\frac{p_n}{q_n}\right|,\left|\gamma-\frac{s_{n-1}}{t_{n-1}}\right| \right\}=\max\left\{\frac{1}{|q_n|\,|q_{n+1}|}, \frac{1}{|t_{n-1}|\,|t_n|}\right\}.
\end{equation}
Comparing \eqref{ex_step1_prop} and \eqref{ex_step1_prop_ultrametric}, we see that it is sufficient to verify the inequality
\begin{equation}
\label{2q_n_verify}
\deg(q_n) - \deg(t_{n-1}) < \deg(t_n) - \deg(q_n) \,.
\end{equation}
One can easily see that for every $n\ge 1$ the left side of \eqref{2q_n_verify} is $2n-1$ and the right side is $2n$.

The second equality in \eqref{bn_cn_example} for $j=n$ is 
equivalent to
\begin{equation}
\label{ex_step2_prop}
\left|\alpha-\frac{p_{n}}{q_{n}}-\frac{s_n}{t_n}\right|<\frac{1}{|t_n|^2}\,.
\end{equation}
Applying the ultrametric property once again, we have
\begin{equation}
\label{ex_step2_prop_ultrametric}
\left|\alpha-\frac{p_{n}}{q_{n}}-\frac{s_n}{t_n}\right|\le
\max\left\{\left|\beta-\frac{p_n}{q_n}\right|,\left|\gamma-\frac{s_{n}}{t_{n}}\right| \right\}=\max\left\{\frac{1}{|q_n|\,|q_{n+1}|}, \frac{1}{|t_{n}|\,|t_{n+1}|}\right\}.
\end{equation}
Comparing \eqref{ex_step2_prop} and \eqref{ex_step2_prop_ultrametric}, we see that it is sufficient to verify the inequality
\begin{equation}
\label{2t_n_verify}
\deg(t_n) - \deg(q_n) < \deg(q_{n+1}) - \deg(t_n)\,.   
\end{equation}
One can easily see that for every $n\ge 1$ the left side of  \eqref{2t_n_verify} is $2n$ and the right side is $2n+1$.
\end{proof}

\paragraph{}
\renewcommand{\abstractname}{Acknowledgements}
\begin{abstract}
	EN wishes to thank Felipe Ram{\'\i}rez for a discussion about the Laurent series analogue of \eqref{thm:hall} during the conference ``Diophantine approximation and related fields'' at York 2025.
\end{abstract}

\bibliographystyle{plain}
\bibliography{refs.bib}

\end{document}